\documentclass[reqno, 12pt]{amsart}
\usepackage{amssymb,amsmath,amsfonts,amsthm,comment,mathrsfs,times,graphicx}
\usepackage[bookmarksnumbered, plainpages]{hyperref}
\usepackage{color}
\usepackage[english]{babel}
\usepackage[all,cmtip]{xy}
\usepackage{lmodern}
\usepackage{enumitem}
\usepackage{geometry}
\usepackage{tikz}
\usetikzlibrary{matrix,arrows}
\usepackage[T2A, T1]{fontenc}
\usepackage[utf8]{inputenc} 
\usepackage{amsmath}

\usepackage{scalerel,stackengine}
\stackMath
\newcommand\reallywidehat[1]{%
	\savestack{\tmpbox}{\stretchto{%
			\scaleto{%
				\scalerel*[\widthof{\ensuremath{#1}}]{\kern-.6pt\bigwedge\kern-.6pt}%
				{\rule[-\textheight/2]{1ex}{\textheight}}
			}{\textheight}%
		}{0.5ex}}%
	\stackon[1pt]{#1}{\tmpbox}%
}
\newcommand\sh[1]{\ensuremath{\mathop{\text{\Large\fontencoding{T2A}\selectfont ш}}#1}}
\newcounter{dummy} \numberwithin{dummy}{section}
\newtheorem{theo}[dummy]{Theorem}
\newtheorem*{theo*}{Theorem}
\newtheorem{coro}[dummy]{Corollary}
\newtheorem{lem}[dummy]{Lemma}
\newtheorem{pro}[dummy]{Proposition}

\newtheorem{rem}[dummy]{Remark}

\newtheorem*{example*}{Exemple}
\newtheorem*{examples*}{Exemples}
\newtheorem*{thm*}{Theorem}
\newtheorem*{definition*}{D\'{e}finition}
\newtheorem*{lem*}{Lemme}
\newtheorem*{prop*}{proposition}
\newcommand{\co}{{\mathrm{Cone}}}

\newcommand{\Z}{{\mathbb{Z}}}
\newcommand{\Q}{{\mathbb{Q}}}

\title[ ]{ Poitou-Tate Duality for totally positive Galois cohomology }
\author[  H. Asensouyis, J. Assim, Z. Boughadi  \& Y. Mazigh]{Hassan Asensouyis, Jilali Assim, Zouhair Boughadi and  Youness Mazigh}
\address{ Moulay Ismail University of Mekn\`{e}s, Faculty of Sciences, Departement of Mathematics, B.P. 11201 Zitoune, 50000 Meknes, Morocco}
\email{\textcolor[rgb]{0.00,0.00,1.00}{hassan$_{-}$asensouyis@yahoo.fr }}
\email{\textcolor[rgb]{0.00,0.00,1.00}{j.assim@umi.ac.ma}}
\email{\textcolor[rgb]{0.00,0.00,1.00}{z.boughadi@edu.umi.ac.ma}}
\email{\textcolor[rgb]{0.00,0.00,1.00}{y.mazigh@umi.ac.ma}}

\begin{document}
	\maketitle
	\renewcommand{\abstractname}{Abstract}
	\begin{abstract}
		In this paper, we establish a Poitou-Tate's global duality for totally positive Galois cohomology. We illustrate this result in the case of the twisted module "à la Tate" $\Z_2(i),$ $i$ integer.
	\end{abstract}
	\section{Introduction}
	Let $F$ be a number field and let $p$ be a rational prime.	For  a finite set $S$ of primes of  $F$ containing the $p$-adic and the infinite primes, we denote by $S_{f}$ the set of finite primes in $S$ and $G_{F,S}$ the Galois group of the maximal algebraic extension $F_{S}$ of $F$ which is unramified outside $S.$\vskip 6pt
	For an odd prime $p$ and $n=1, 2,$ the global Poitou-Tate duality (e.g. \cite[5.1.6, p.114]{Nekovar06}) states that there is a perfect pairing
	\begin{equation}\label{Poitou-Tate pairing}
		\begin{array}{ccccc}
			\sh^{n}_{S_{f}}(M) & \times & \sh^{3-n}_{S_{f}}(M^{\ast}) & \xymatrix@=1.5pc{ \ar[r]&} &
			\Q_{p}/\Z_{p},
		\end{array}
	\end{equation}
	where
	\begin{equation*}	
		\mbox{$\sh^{n}_{S_{f}}(M)$}:=\xymatrix@=1.5pc{\ker( H^{n}(G_{F,S},M)\ar[r]& \displaystyle{\bigoplus_{v\in	S_{f}}H^{n}(F_{v},M)})}
	\end{equation*}
	and  $(.)^{\ast}$ means the Kummer dual.
	\vskip 7pt
	In the case $p=2,$ the non-triviality of the cohomology groups of the local absolute Galois group at real places  leads to several complications. To control these contributions from real infinite places, several authors (e.g.\,\cite{CKPS, rognes, Mazigh, AAM21}) use a slight variant of Galois cohomology, the so-called totally positive Galois cohomology introduced by Kahn in \cite{Ka93} after ideas of Milne \cite{Milne}.\vskip 6pt
	In this paper we establish a global Poitou-Tate duality for totally positive Galois cohomology.
	For $n=1, 2,$ let
	$\sh_{S_{f}}^{n,+}(M)$  be the kernels of the localization maps
	\begin{equation*}
		\mbox{$\sh_{S_{f}}^{n,+}(M)$}:=\xymatrix@=1.5pc{\ker(
			H^{n}_{+}(G_{F,S},M)\ar[r]& \bigoplus_{v\in S_{f}}H^{n}(F_{v},M)),}
	\end{equation*}
	where
	$H^{j}_{+}(.,.)$ denotes  the $j$-th totally positive Galois
	cohomology group (Section \ref{Totally positive Galois cohomology}). The following theorem summarizes the main result of this paper (see Theorem \ref{theorem1}):
	\begin{theo*} Let $F$ be a number field and let $M$ be a compact or discrete $\Z_{2}[[G_{F,S}]]$-module.
		\begin{enumerate}[label=(\roman*)]
			\item 	There is a perfect pairing
			\begin{equation*}
				\begin{array}{ccccc}
					\sh_{S_{f}}^{2,+}(M) & \times & \sh_{S_{f}}^{1}(M^{\ast}) & \xymatrix@=1.5pc{ \ar[r]&} &
					\Q_{2}/\Z_{2}.
				\end{array}
			\end{equation*}
			\item We have an exact sequence	
			\begin{equation*}
				\xymatrix@=2pc{0\ar[r]& Z(M) \ar[r]&
					\sh_{S_{f}}^{1,+}(M)\ar[r]& \sh_{S_{f}}^{2}(M^{\ast})^{\vee}\ar[r]&0,}
			\end{equation*}
			where an explicit description of the kernel $Z(M)$ is given in Lemma \ref{Lemma Z(M)}.
		\end{enumerate}
		The superscripts $(.)^{\ast}$ and $(.)^{\vee},$ respectively denote the Kummer dual and the Pontryagin dual.
	\end{theo*}
	For an integer $i$, let $M=\Z_{2}(i)$ be the twisted module "à la Tate" of the ring of dyadic integers $\Z_{2}.$ We have the following exact sequence
	\begin{equation*}
		\xymatrix@=2pc{0\ar[r]& Z(\Z_{2}(i))\ar[r]& \sh_{S_{f}}^{1,+}(\mathbb{Z}_{2}(i))\ar[r]& \sh_{S_{f}}^{2}(\mathbb{Q}_{2}/\mathbb{Z}_{2}(1-i))^{\vee}\ar[r]&0,}
	\end{equation*}
	where
	\begin{equation*}
		Z(\Z_{2}(i))\cong \begin{cases}
			\qquad 0, \qquad \qquad\qquad \quad \qquad\quad\, \mbox{ if } i=0;\\\
			\underset{v \: is \: complex}{\bigoplus} \Z_{2}(i) \underset{v \: is \: real}{\bigoplus} 2\Z_{2}(i) ,\quad \mbox{ if $i\neq 0$ is even;} \\\
			\underset{v \: is \: complex}{\bigoplus} \Z_{2}(i), \quad\qquad \,\quad\qquad \mbox{ if $i$ is odd.}
		\end{cases}	
	\end{equation*}
	
	As an application, for a number field $F$ and $\mathcal{O}_{F,S}$ it's ring of $S$-integers, we realize the positive \'{e}tale wild kernel (\cite[Definition 2.2]{AAM21}) $WK_{2i-2}^{\mbox{\'{e}t},+}\mathcal{O}_{F,S}:=\sh_{S_{f}}^{2,+}(\mathbb{Z}_{2}(i)),$ for $i\geq 2$ an integer, as an Iwasawa module (Proposition \ref{Iwasawa description}). In particular, we get that the group $WK_{2i-2}^{\mbox{\'{e}t},+}\mathcal{O}_{F,S}$ is independent of the set $S$ containing the infinite and dyadic places of $F.$

	\section{Some homological algebra}\label{sectiontwo}
	We fix an abelian category and work with the corresponding category of complexes. For a complex $X=(X^{i},d^{i}_{X})_{i\in\Z}$ and an integer $n\in\Z,$ let $X[n]$ denote the complex given by the objects
	$(X[n])^{i}=X^{i+n}$ and the differentials $d^{i}_{X[n]}=(-1)^{n}d^{i+n}_{X}.$ For a morphism of complexes $\xymatrix@=1.5pc{ X\ar[r]^-{u}& Y,}$ the mapping  cone corresponding to $u$ is the complex 	
	\begin{equation*}
		\mathrm{Cone}(u):=Y\oplus X[1]
	\end{equation*}	
	with the differential
	\begin{equation*}
		d^{i}_{\mathrm{Cone}(u)}=\begin{pmatrix}
			d^{i}_{Y}& u^{i+1}\\
			0& -d^{i+1}_{X}
		\end{pmatrix}:\;\xymatrix@=2pc{ Y^{i}\oplus X^{i+1}\ar[r]& Y^{i+1}\oplus X^{i+2}.}
	\end{equation*}
	The distinguished triangle corresponding to $u$ is
	\begin{equation}\label{canonical form}
		\xymatrix@=2pc{X\ar[r]^-{u}& Y\ar[r]^-{j}& \mathrm{Cone}(u)\ar[r]^-{-\pi}&X[1],}
	\end{equation}
	where $j$ and  $\pi$ are  the canonical  injection and  projection  respectively (\cite[chap.II, \S 0, p.180]{Milne}). By definition, a distinguished triangle is isomorphic (in the derived category) to the form \eqref{canonical form}, or equivalently to
	\begin{equation*}
		\xymatrix@=2pc{ \mathrm{Cone}(u)[-1]\ar[r]^-{\pi[-1]}&X\ar[r]^-{u}& Y\ar[r]^-{j}& \mathrm{Cone}(u).}
	\end{equation*}
	Furthermore, a distinguished triangle
	\begin{equation*}
		\xymatrix@=2pc{ X\ar[r]& Y\ar[r]&Z\ar[r]&X[1]}
	\end{equation*}
	gives rise to a long exact sequence in cohomology
	\begin{equation*}
		\xymatrix@=1.5pc{\cdots\ar[r]& H^{r}(X)\ar[r]& H^{r}(Y)\ar[r]&H^{r}(Z)\ar[r]&H^{r+1}(X)\ar[r]&\cdots}
	\end{equation*}
	\noindent Recall also that, for a given commutative diagram
	\begin{equation*}
		\xymatrix@=1.5pc{X\ar[r]^-{u}\ar[d]_{f}& Y\ar[d]_{g}\\
			X_{1}\ar[r]^-{v}& Y_{1}}
	\end{equation*}
	we have a morphism of mapping cones (\cite[\S 3.1, p. 66]{Verdier})
	\begin{equation*}
		\xymatrix@=2pc{(g,f[1]):\; \co(u)\ar[r]&\co(v)}
	\end{equation*}
	where
	\begin{equation*}
		(g,f[1])^{i}=\begin{pmatrix}
			g^{i}& 0\\
			0& f^{i+1}
		\end{pmatrix}:\; \xymatrix@=2pc{Y^{i}\oplus X^{i+1}\ar[r]& Y_{1}^{i}\oplus X_{1}^{i+1}.}
	\end{equation*}
	The following proposition gives two interesting results about the mapping cone which will be useful in the sequel.
	\begin{pro}\label{Pro cone}	$ $
		\begin{enumerate}
			\item For any maps $u: X\longrightarrow Y$ and $v: Y\longrightarrow Z$ of  complexes, there is a distinguished triangle
			\begin{equation*}
				\xymatrix@=3pc{\co(v)[-1]\ar[r]&\co(u)\ar[r]^{(v,id_{X}[1])}&\co(v\circ u)\ar[r]^{(id_{Z}, u[1])}&\co(v).}
			\end{equation*}
			\item An exact commutative diagram of complexes
			\begin{equation*}
				\xymatrix@=3pc{0\ar[r]& X\ar[r]^-{f}\ar[d]^-{\alpha}& Y\ar[r]^-{g}\ar[d]^-{\beta}& Z\ar[r]\ar[d]^-{\gamma}&0\\
					0\ar[r]& X_{1}\ar[r]^-{f_{1}}& Y_{1}\ar[r]^-{g_{1}}& Z_{1}\ar[r]&0}
			\end{equation*}
			gives rise to a distinguished triangle
			\begin{equation*}
				\xymatrix@=2pc{\co(\alpha)\ar[r]&\co(\beta)\ar[r]&\co(\gamma)\ar[r]&\co(\alpha)[1].}
			\end{equation*}
		\end{enumerate}
	\end{pro}
	\begin{proof}
		see \cite[chap.II, \S 0, Prop. 0.10]{Milne}
	\end{proof}
	\begin{pro}\label{pro coker and cone}
		Let
		\begin{equation*}
			\xymatrix@=2pc{0\ar[r]& X\ar[r]^-{u}& Y\ar[r]^-{v}& Z\ar[r]& 0}
		\end{equation*}
		be a  short exact sequence of complexes. Then, the  maps
		\begin{equation*}\label{coneZ}
			\begin{array}{ccccccccc}
				q: & \mathrm{Cone}(u)& \xymatrix@=2pc{\ar[r]&}& Z\\
				& (y,x)&\xymatrix@=2pc{\ar@{|->}[r]&}& v(y)
			\end{array}
			\quad\mbox{and}\quad \begin{array}{ccccccccc}
				\ell : & X& \xymatrix@=2pc{\ar[r]&}& \mathrm{Cone}(v)[-1]\\
				& x&\xymatrix@=2pc{\ar@{|->}[r]&}& (0,u(x))
			\end{array}
		\end{equation*}
		are quasi-isomorphisms.
	\end{pro}
	\begin{proof}
		Consider the short exact sequences
		\begin{equation*}
			\xymatrix@=3pc{0\ar[r]& \mathrm{Cone}(\mathrm{id}_{X})\ar[r]^-{(u,\mathrm{id}_{X[1]})}& \mathrm{Cone}(u)\ar[r]^-{q}& Z\ar[r]&0}
		\end{equation*}
		and
		\begin{equation*}
			\xymatrix@=3pc{0\ar[r]& X\ar[r]^-{\ell}& \mathrm{Cone}(v)[-1]\ar[r]^-{(\mathrm{id}_{Z[-1]},v)}& \mathrm{Cone}(\mathrm{id}_{Z[-1]})\ar[r]&0.}
		\end{equation*}
		Since  $\mathrm{Cone}(\mathrm{id}_{A})$ is  acyclic for any complex $A,$  the long exact sequences of cohomology imply that $q$ and $\ell$ are quasi-isomorphisms.
	\end{proof}
	Now, given a commutative cubic diagram of complexes:
	\begin{equation*}
		\xymatrix@=1pc{ C\ar[rrrr]^-{f}\ar[dr]^-{id}\ar[ddddd]^-{id}&  &  &  & X\ar[dr]^-{u}\ar[ddddd]^-{v}\\
			& C \ar[rrrr]^-{g}\ar[ddddd]^-{id}&  &  &  & Y\ar[ddddd]^-{\hat{v}}\\
			&  &   &   &\\
			&  &  &   &  \\
			&  &  &  &   \\
			C\ar[rrrr]^-{f_{1}}\ar[dr]^-{id}&  &  &  &  X_{1}\ar[dr]^-{u_{1}}\\
			&  C\ar[rrrr]^-{g_{1}}&  &  &  &  Y_{1}}
	\end{equation*}
	we have the following proposition.
	\begin{pro}\label{pro cub}
		The above commutative cubic diagram gives the following commutative diagram
		\begin{equation*}
			\xymatrix@=3pc{\mathrm{Cone}(f)\ar[r]^-{(u,id_{C}[1])}\ar[d]^-{(v,id_{C}[1])}& \mathrm{Cone}(g)\ar[r]^-{(id_{Y},f[1])}\ar[d]^-{(\hat{v},id_{C}[1])}& \mathrm{Cone}(u)\ar[d]^-{(\hat{v},v[1])}\ar[r]& \mathrm{Cone}(f)[1]\ar[d]\\
				\mathrm{Cone}(f_{1})\ar[r]^-{(u_{1},id_{C}[1])}\ar[d]^-{(id_{X_{1}},f[1])}   & \mathrm{Cone}(g_{1})\ar[r]^-{(id_{Y_{1}},f_{1}[1])}\ar[d]^-{(id_{Y_{1}},g[1])}    &  \mathrm{Cone}(u_{1})\ar[r]  & \mathrm{Cone}(f_{1})[1]\\
				\mathrm{Cone}(v)\ar[r]^-{(u_{1},u[1])}\ar[d]& \mathrm{Cone}(\hat{v})\ar[d]&   &  \\
				\mathrm{Cone}(f)[1]\ar[r]& \mathrm{Cone}(g)[1]&    &    }
		\end{equation*}
		where the two first rows and columns are distinguished triangles.
	\end{pro}
	\begin{proof}
		Obviously, the commutativity is deduced from the commutative cubic diagram. Further, Proposition \ref{Pro cone} ($1$) shows that the two first rows and columns are distinguished triangles.  	
	\end{proof}
	
	\section{Poitou-Tate duality for totally positive Galois cohomology}\label{Totally positive Galois cohomology}
	For a field $K,$ let $K^{sep}$ denote a fixed separable closure of $K$ and $G_{K}=\mathrm{Gal}(K^{sep}/K).$ Let $F$ be a number field and let $S$ be a finite set of primes of $F$ containing the set $S_{2}$ of dyadic primes  and the set $S_{\infty}$ of archimedean primes. For a place $v$ of $F,$ we denote by $F_{v}$ the completion of $F$ at $v.$ Notice that for any \textbf{infinite} place $v$ of $F,$ a fixed extension $F^{sep}\hookrightarrow F_{v}^{sep}$ of the embedding  $F\hookrightarrow F_{v}$  defines a continuous homomorphism
	$\xymatrix@=1pc{G_{F_{v}}\ar@{^{(}->}[r]& G_{F,S},}$ where $G_{F,S}$ is the Galois group of the maximal algebraic extension $F_{S}$ of $F$ which is unramified outside $S.$
	\vskip 6pt
	For a discrete or a compact  $\Z_{2}[[\mathrm{G_{F,S}}]]$-module $M,$ we write $M_{+}$
	for the cokernel of the map
	\begin{equation*}
		\xymatrix@=2pc{ M\ar[r]&
			\displaystyle{\bigoplus_{v\mid\infty}\mathrm{Ind}^{G_{F,S}}_{G_{F_{v}}}}M,}
	\end{equation*}
	where
	$\mathrm{Ind}^{G_{F,S}}_{G_{F_{v}}}M$ denotes the induced module. Thus, we have an exact sequence
	\begin{equation}\label{defmpositive}
		\xymatrix@=2pc{0\ar[r]&M\ar[r]&\displaystyle{\bigoplus_{v\mid
					\infty}\mathrm{Ind}^{G_{F,S}}_{G_{F_{v}}}}M\ar[r]&M_{+}\ar[r]&0.}
	\end{equation}
	Following \cite[\S 5]{Ka93} and \cite[Section 2]{CKPS}, we define the $n$-$\mathrm{th}$ totally positive Galois cohomology group $H^{n}_{+}(G_{F,S},M)$ of $M$ by
	\begin{equation}\label{deft positive}
		H^{n}_{+}(G_{F,S},M):=H^{n-1}(G_{F,S},M_{+}),\; n\in\Z.
	\end{equation}
	
	Let us give an equivalent definition of the totally positive
	Galois cohomology in terms of the mapping cone.\\
	By \cite[Proposition 3.4.2, p. 83]{Nekovar06},  the exact sequence (\ref{defmpositive})
	gives rise to the exact  sequence of complexes of $\Z_{2}$-modules
	\begin{equation*}
		\xymatrix@=2pc{0\ar[r]& C^{\bullet}_{cont}(G_{F,S},M)\ar[r]^-{i^{\bullet}}& \bigoplus_{v\mid\infty} C^{\bullet}_{cont}(G_{F,S},\mathrm{Ind}^{G_{F,S}}_{G_{F_{v}}}M)\ar[r]& C^{\bullet}_{cont}(G_{F,S},M_{+})\ar[r]&0}
	\end{equation*}
	where $C^{\bullet}_{cont}(G_{F,S},.)$ denotes  the complex of continuous cochains \cite[Definition 3.4.1.1]{Nekovar06}.\\
	On the one hand,  Proposition \ref{pro coker and cone} shows that the complex $\mathrm{Cone}(i^{\bullet})$  is quasi-isomorphic to the complex
	$C^{\bullet}_{cont}(G_{F,S},M_{+}).$
	On the other hand, for any infinite prime $v,$ Shapiro's lemma applied to $M$ gives a quasi-isomorphism
	\begin{equation*}
		\xymatrix@=2pc{Sh_{v}: C^{\bullet}_{cont}(G_{F,S},\mathrm{Ind}^{G_{F,S}}_{G_{F_{v}}}M)\ar[r]& C^{\bullet}_{cont}(G_{F_{v}},M).}
	\end{equation*}
	It follows that
	\begin{equation*}
		\xymatrix@=2pc{Sh_{\infty}:=\bigoplus_{v\mid \infty} Sh_{v}: \bigoplus_{v\mid \infty} C^{\bullet}_{cont}(G_{F,S},\mathrm{Ind}^{G_{F,S}}_{G_{F_{v}}}M)\ar[r]& \bigoplus_{v\mid \infty} C^{\bullet}_{cont}(G_{F_{v}},M)}
	\end{equation*}
	is a quasi-isomorphism and that
	$\mathrm{Cone}(Sh_{\infty})$ is acyclic.\\
	Further, by Proposition \ref{Pro cone} ($1$), there is a distinguished triangle
	\begin{equation*}
		\xymatrix@=2pc{\mathrm{Cone}(Sh_{\infty})[-1]\ar[r]& \mathrm{Cone}(i^{\bullet})\ar[r]&
			\mathrm{Cone}(Sh_{\infty}\circ i^{\bullet})\ar[r]& \mathrm{Cone}(Sh_{\infty}).}
	\end{equation*}
	Writing the long exact sequence of cohomology, we obtain
	that  $\mathrm{Cone}(Sh_{\infty}\circ i^{\bullet})$ is quasi-isomorphic to $\mathrm{Cone}(i^{\bullet}).$ We deduce that the complexes  $\mathrm{Cone}(Sh_{\infty}\circ i^{\bullet})$ and $C^{\bullet}_{cont}(G_{F,S},M_{+})$  are quasi-isomorphic. Then, for all $n\geq 0$ 	
	\begin{equation*}
		H^{n}(C^{\bullet}_{cont}(G_{F,S},M_{+}))=H^{n}(\mathrm{Cone}(Sh_{\infty}\circ i^{\bullet})).
	\end{equation*}	
	Using  (\ref{deft positive}), we get the following description of the totally positive Galois cohomology in terms of the mapping cone.
	\begin{lem}\label{defposigal}
		For any integer $n\geq 0,$ we have
		\begin{equation*}
			H^{n}_{+}(G_{F,S},M)=H^{n-1}(\mathrm{Cone}(Sh_{\infty}\circ i^{\bullet}))\cong
			H^{n}(\mathrm{Cone}(Sh_{\infty}\circ i^{\bullet})[-1]).
		\end{equation*}
	\end{lem}	
	In the next proposition, we prove an analogue of the Poitou-Tate exact sequence for the prime $p=2$ involving the totally positive Galois cohomology groups $H^{i}_{+}(.),$ $i=1, 2.$ The proof uses  the methods of \cite{Nekovar06}.\\
	For each prime $v\in S,$ fix  an embedding $F^{sep}\hookrightarrow F_{v}^{sep}$ extending $F\hookrightarrow F_{v}.$ This defines a continuous homomorphism	
	$ \xymatrix@=1pc{G_{F_{v}}\ar@{^{(}->}[r]&G_{F}\ar@{->>}[r]& G_{F,S},}$ hence a 'restriction' map (\cite[p.113]{Nekovar06})
	\begin{equation*}\label{restriction map}
		\xymatrix@=2pc{	res_{v}:\; C^{\bullet}_{cont}(G_{F,S},M)\ar[r]& C^{\bullet}_{cont}(G_{F_{v}},M).}
	\end{equation*}	
	Let $res_{S}:=\bigoplus_{v\in S}res_{v}$ and $\pi_{\infty}$ be the projection map $$\xymatrix@=1.5pc{	\bigoplus_{v\in S}C^{\bullet}_{cont}(G_{F_{v}},M)\ar[r]^-{\pi_{\infty}}& \bigoplus_{v\mid \infty}C^{\bullet}_{cont}(G_{F_{v}},M).}$$
	By Proposition \ref{Pro cone} ($1$),
	the composite morphism of complexes
	\begin{equation*}
		\xymatrix@=3pc{C^{\bullet}_{cont}(G_{F,S},M)\ar[r]^-{res_{S}}&\bigoplus_{v\in S}C^{\bullet}_{cont}(G_{F_{v}},M)\ar[r]^-{\pi_{\infty}}& \bigoplus_{v\mid \infty}C^{\bullet}_{cont}(G_{F_{v}},M),}
	\end{equation*}	
	induces a distinguished triangle	
	\begin{equation*}
		\xymatrix@=2pc{\mathrm{Cone}(\pi_{\infty})[-1]\ar[r]& \mathrm{Cone}(res_{S})\ar[r]& \mathrm{Cone}(\pi_{\infty}\circ res_{S})\ar[r]& \mathrm{Cone}(\pi_{\infty}),}
	\end{equation*}
	This implies that there is a  long exact sequence of cohomology groups
	\begin{equation}\label{longcone}
		\xymatrix@=0.5pc{\cdots\ar[r]& H^{n}(\mathrm{Cone}(res_{S}))\ar[r]& H^{n}(\mathrm{Cone}(\pi_{\infty}\circ res_{S}))\ar[r]& H^{n}(\mathrm{Cone}(\pi_{\infty}))\ar[r]& H^{n+1}(\mathrm{Cone}(res_{S}))\ar[r]&\cdots}
	\end{equation}	
	Let us compute  the cohomology groups involved in this long exact sequence in terms of  $G_{F,S}$-cohomology and $G_{F_{v}}$-cohomology.\\
	Firstly, remark that for an infinite place $v,$ the morphism $res_{v}$ is the composite
	\begin{equation*}
		\xymatrix@=2pc{ C^{\bullet}_{cont}(G_{F,S},M)\ar[r]&  C^{\bullet}_{cont}(G_{F,S},\mathrm{Ind}^{G_{F,S}}_{G_{F_{v}}}M)\ar[r]^-{Sh_{v}}& C^{\bullet}_{cont}(G_{F_{v}},M).}
	\end{equation*}
	Then the composite
	\begin{equation*}
		\xymatrix@=2pc{ C^{\bullet}_{cont}(G_{F,S},M)\ar[r]^-{i^{\bullet}}&  \bigoplus_{v\mid\infty} C^{\bullet}_{cont}(G_{F,S},\mathrm{Ind}^{G_{F,S}}_{G_{F_{v}}}M)\ar[r]^-{Sh_{\infty}}&\bigoplus_{v\mid\infty} C^{\bullet}_{cont}(G_{F_{v}},M)}
	\end{equation*}
	is exactly the morphism $\pi_{\infty}\circ res_{S}.$ Hence, by Lemma \ref{defposigal} we get
	\begin{equation}\label{description totally positive cone}
		H^{n+1}_{+}(G_{F,S},M) \cong H^{n}(\mathrm{Cone}(\pi_{\infty}\circ res_{S}))
	\end{equation}
	for all $n\in\Z.$\\
	Secondly, by Proposition \ref{pro coker and cone}, the exact sequence
	\begin{equation*}
		\xymatrix@=1.5pc{0\ar[r]& \bigoplus_{v\in S_{f}}C^{\bullet}_{cont}(G_{F_{v}},M)\ar[r]& \bigoplus_{v\in S} C^{\bullet}_{cont}(G_{F_{v}},M)\ar[r]^-{\pi_{\infty}}&
			\bigoplus_{v\mid \infty}C^{\bullet}_{cont}(G_{F_{v}},M)\ar[r]&0}
	\end{equation*}
	gives rise to a  quasi-isomorphism between $\bigoplus_{v\in S_{f}}C^{\bullet}_{cont}(G_{F_{v}},M)$ and $\mathrm{Cone}(\pi_{\infty})[-1].$	Hence, for all $n\in\Z,$
	\begin{equation}\label{cone pi}
		H^{n}(\mathrm{Cone}(\pi_{\infty}))\cong H^{n+1}(\bigoplus_{v\in S_{f}}C^{\bullet}_{cont}(G_{F_{v}},M))\cong \bigoplus_{v\in S_{f}}H^{n+1}(F_{v},M).
	\end{equation}
	Finally, consider the map
	\begin{equation*}
		\xymatrix@=2pc{	\widehat{res}_{S}:\; C^{\bullet}_{cont}(G_{F,S},M)\ar[r]& \bigoplus_{v\in S_{f}}C^{\bullet}_{cont}(G_{F_{v}},M)\oplus \bigoplus_{v\mid \infty}\hat{C}^{\bullet}_{cont}(G_{F_{v}},M),}
	\end{equation*}
	where  $\hat{C}^{\bullet}_{cont}(G_{F_{v}},M)$  is the complete Tate cochain complex,	$\widehat{res}_{S}=\oplus_{v\in S_{f}} res_{v}\bigoplus \oplus_{v\mid\infty}\widehat{res}_{v}$ and for $v\mid\infty$  the map $\widehat{res}_{v}$  is given by
	\begin{equation*}
		\xymatrix@=2pc{	\widehat{res}_{v}:\; C^{\bullet}_{cont}(G_{F,S},M)\ar[r]^-{res_{v}}& C^{\bullet}_{cont}(G_{F_{v}},M)\ar@{^{(}->}[r]^-{\tau_{v}}& \hat{C}^{\bullet}_{cont}(G_{F_{v}},M),}
	\end{equation*}
	$\tau_{v}$ being the canonical injection. Recall that the continuous cohomology group  with compact support $\hat{H}^{n}_{c,cont}(G_{F,S},M)$ is defined by
	\begin{equation*}
		\hat{H}^{n}_{c,cont}(G_{F,S},M):=H^{n}(\mathrm{Cone}(\widehat{res}_{S})[-1])
	\end{equation*}
	cf.\,\cite[ch.5, p.132]{Nekovar06}.
	
	In the sequel, we adopt the following usual notations:
	\begin{itemize}
		\item [$(.)^{\ast}:$] the Kummer dual $\mathrm{Hom}(.,\mu_{2^{\infty}}),$ where $\mu_{2^{\infty}}$ is the group of all roots of unity of $2$-power order;
		\item [$(.)^{\vee}:$] the Pontryagin dual $\mathrm{Hom}(.,\Q_{2}/\Z_{2}).$
	\end{itemize}
	Notice that:
	\begin{enumerate}[label=(\roman*)]
		\item For all $n\geq 1,$    $H^{n}(\mathrm{Cone}(res_{S}))=H^{n}(\mathrm{Cone}(\widehat{res}_{S}))=\hat{H}^{n+1}_{c,cont}(G_{F,S},M).$ In particular
		$H^{n}(\mathrm{Cone}(res_{S}))\cong H^{2-n}(G_{F,S},M^{\ast})^{\vee}$ as a consequence of \cite[Proposition 5.7.4.]{Nekovar06}.
		\item $	H^{0}(\mathrm{Cone}(\widehat{res}_{S}))= \hat{H}^{1}_{c,cont}(G_{F,S},M)\cong H^{2}(G_{F,S},M^{\ast})^{\vee}$ (\cite[Proposition 5.7.4.]{Nekovar06}.
	\end{enumerate}
	\vskip 6pt
	From \eqref{longcone} and the above results, we deduce the following proposition:
	\begin{pro}\label{exact sequence of finie places} Let $M$ be a finite $\Z_{2}[[G_{F,S}]]$-module and let $S_{f}$ denote the set of finite places in $S.$ Then there is a
		long exact sequence
		\begin{equation*}
			\begin{tikzpicture}[descr/.style={fill=white,inner sep=2pt}]
				\matrix (m) [
				matrix of math nodes,
				row sep=1.4em,
				column sep=1.1em,
				text height=1.5ex, text depth=0.0ex
				]
				{0&\bigoplus_{v\in S_{f}}H^{0}(F_{v},M)&  H^{0}(\co(res_{S})) & H^{1}_{+}(G_{F,S},M) & \bigoplus_{v\in S_{f}}H^{1}(F_{v},M) & \\
					& H^{1}(G_{F,S},M^{\ast})^{\vee}& H^{2}_{+}(G_{F,S},M) &  \bigoplus_{v\in
						S_{f}}H^{2}(F_{v},M) & H^{0}(G_{F,S},M^{\ast})^{\vee}& 0. \\
				};
				
				\path[overlay,->, font=\scriptsize,>=latex]
				(m-1-1) edge (m-1-2)
				(m-1-2) edge (m-1-3)
				(m-1-3) edge (m-1-4)
				(m-1-4) edge (m-1-5)
				(m-1-5) edge[out=355,in=175,black] node[yshift=0.3ex] {} (m-2-2)
				(m-2-2) edge (m-2-3)
				(m-2-3) edge (m-2-4)
				(m-2-4) edge (m-2-5)
				(m-2-5) edge (m-2-6)
				;
			\end{tikzpicture}
		\end{equation*}	
	\end{pro}
	\begin{rem} Proposition \ref{exact sequence of finie places} above is a slight generalization of \cite[Proposition 2.6]{Mazigh}. However, a certain argument concerning the continuous cohomology with compact support in \cite[\S 2.2, (4) p.6 ]{Mazigh} turns out to be incorrect. In \cite[\S 2.2, (4) p.6 ]{Mazigh} the  author claimed that 	$H^{n}(G_{F,S},M_{S})=\hat{H}^{n+1}_{c,cont}(G_{F,S},M),$ where $M_{S}$ is  the  cokernel of the canonical map
		$$\xymatrix@=2pc{M\ar[r]& \bigoplus_{v\in S} \mathrm{Ind}^{G_{F}}_{G_{F_{v}}}M.}$$
	This is not always true but  the results of \cite{Mazigh} remain unchanged using Proposition \ref{exact sequence of finie places} above.
	\end{rem}
	
	Let us recall  the local duality Theorem (e.g.\,\cite[Corollary
	2.3, p.34]{Milne}). For $n=0,1,2$ and for every place $v$ of $F,$ the cup products
	\begin{equation}\label{local duality}
		\begin{array}{cccccc}
			H^{n}(F_{v},M) & \times& H^{2-n}(F_{v},M^{\ast}) & \xymatrix@=1.5pc{
				\ar[r]&} &
			H^{2}(F_{v},\mu_{2^{\infty}})\simeq \Q_{2}/\Z_{2},&
			\mbox{if $v$ is finite}\\
			&   &    &   &    \\
			\widehat{H}^{n}(F_{v},M) & \times& \widehat{H}^{2-n}(F_{v},M^{\ast}) & \xymatrix@=1.5pc{ \ar[r]&} &
			H^{2}(F_{v},\mu_{2^{\infty}}),&
			\mbox{if $v$ is infinite}
		\end{array}
	\end{equation}
	are perfect pairings, where $\widehat{H}^{n}(F_{v},.)$
	is the Tate cohomology groups.\\
	For a discrete or a compact  $\Z_{2}[[\mathrm{G_{F,S}}]]$-module $M$ and $n=1,2,$ we define the groups  $\sh_{S_{f}}^{n}(M)$   and
	$\sh_{S_{f}}^{n,+}(M)$ to be the kernels of the localization maps
	\begin{equation*}
		\mbox{$\sh_{S_{f}}^{n}(M)$}:=\xymatrix@=1.5pc{\ker( H^{n}(G_{F,S},M)\ar[r]& \bigoplus_{v\in
				S_{f}}H^{n}(F_{v},M)),}
	\end{equation*}
	and
	\begin{equation*}
		\mbox{$\sh_{S_{f}}^{n,+}(M)$}:=\xymatrix@=1.5pc{\ker(
			H^{n}_{+}(G_{F,S},M)\ar[r]& \bigoplus_{v\in S_{f}}H^{n}(F_{v},M)).}
	\end{equation*}\vskip 6pt
	Here we are interested in the study of the analogue of the pairing \eqref{Poitou-Tate pairing} for the groups $\sh_{S_{f}}^{n ,+}(M)$ with $n=1, 2.$\vskip 6pt
	Let's start with the case $n=2.$ We have the following proposition:
	\begin{pro}
		Let $M$ be a discrete or a compact $\Z_{2}[[G_{F,S}]]$-module.
		There is a perfect pairing
		\begin{equation*}
			\begin{array}{ccccc}
				\sh_{S_{f}}^{2,+}(M) & \times & \sh_{S_{f}}^{1}(M^{\ast}) & \xymatrix@=1.5pc{ \ar[r]&} &
				\Q_{2}/\Z_{2}.
			\end{array}
		\end{equation*}
	\end{pro}
	\begin{proof}
		The exact sequence of Proposition \ref{exact sequence of finie
			places} and the definition of $\sh_{S}^{2,+}(M)$  give  the  exact sequence
		\begin{equation*}
			\xymatrix@=2pc{\bigoplus_{v\in S_{f}}H^{1}(F_{v},M)\ar[r]& H^{1}(G_{F,S},M^{\ast})^{\vee}\ar[r]&
				\sh_{S_{f}}^{2,+}(M)\ar[r]&0.}
		\end{equation*}
		Dualizing  this exact sequence and using the local duality $(\ref{local duality}),$ we get
		\begin{equation*}
			\mbox{$\sh_{S_{f}}^{2,+}(M)^{\vee} \cong  \sh_{S_{f}}^{1}(M^{\ast}).$}
		\end{equation*}
	\end{proof}
	The case $n=1$ is more complicated, it requires more than a simple manipulation of the long exact sequence of Proposition \ref{exact sequence of finie places}. First of all,  let's  start by recalling the morphisms already defined in this section
	\begin{eqnarray*}
		\xymatrix@=1.5pc{ \pi_{\infty}:\,	\bigoplus_{v\in S}C^{\bullet}_{cont}(G_{F_{v}},M)\ar[r] & \bigoplus_{v\mid \infty}C^{\bullet}_{cont}(G_{F_{v}},M),\\
			res_{S}:\, \bigoplus_{v\in S}C^{\bullet}_{cont}(G_{F_{v}},M)\ar[r] & \bigoplus_{v\in S_{f}}C^{\bullet}_{cont}(G_{F_{v}},M) \oplus \bigoplus_{v\mid \infty}C^{\bullet}_{cont}(G_{F_{v}},M) \, and \\
			\widehat{res}_{S}:\; C^{\bullet}_{cont}(G_{F,S},M)\ar[r]& \bigoplus_{v\in S_{f}}C^{\bullet}_{cont}(G_{F_{v}},M)\oplus \bigoplus_{v\mid \infty}\hat{C}^{\bullet}_{cont}(G_{F_{v}},M).}
	\end{eqnarray*}
	In addition, we consider the projection map $\hat{\pi}_{\infty}$: $\xymatrix@=1.5pc{	\bigoplus_{v\in S}C^{\bullet}_{cont}(G_{F_{v}},M)\ar[r]& \bigoplus_{v\mid \infty}\hat{C}^{\bullet}_{cont}(G_{F_{v}},M).}$
	We have the following proposition:
	\begin{pro}
		Let $M$ be a discrete or compact $\Z_{2}[[G_{F,S}]]$-module. We have an exact sequence	
		\begin{equation*}
			\xymatrix@=2pc{0\ar[r]& Z(M) \ar[r]&
				\sh_{S_{f}}^{1,+}(M)\ar[r]& \sh_{S_{f}}^{2}(M^{\ast})^{\vee}\ar[r]&0,}
		\end{equation*}
		where
		\begin{equation*}
			Z(M):=\ker(H^{0}(\mathrm{Cone}(\pi_{\infty}\circ res_{S}))\longrightarrow H^{0}(\mathrm{Cone}(\hat{\pi}_{\infty}\circ\widehat{res}_{S}))).
		\end{equation*}
	\end{pro}
	\begin{proof}
	 Consider the commutative exact diagram
		\begin{equation*}
			\xymatrix@=2pc{0\ar[r]& C_{f}\ar[r]\ar@{=}[d]& C_{f}\oplus C_{\infty}\ar[r]^-{\pi_{\infty}}\ar[d]^-{(id,\tau_{\infty})}& C_{\infty}\ar[r]\ar[d]^-{\tau_{\infty}}&0\\
				0\ar[r]& C_{f}\ar[r]& C_{f}\oplus \hat{C}_{\infty}\ar[r]^-{\hat{\pi}_{\infty}}& \hat{C}_{\infty}\ar[r]&0.}
		\end{equation*}
		where \begin{equation*}
			C_{f}:=\bigoplus_{v\in S_{f}}C^{\bullet}_{cont}(G_{F_{v}},M),\quad C_{\infty}:=\bigoplus_{v\mid\infty}C^{\bullet}_{cont}(G_{F_{v}},M),\;\hat{C}_{\infty}:=\bigoplus_{v\mid\infty}\hat{C}^{\bullet}_{cont}(G_{F_{v}},M)
		\end{equation*}
		and  $\tau_{\infty}:=\oplus_{v\mid \infty}\tau_{v}.$ Since $\co(id)$ is acyclic, we obtain, by Propositions \ref{Pro cone} and \ref{pro coker and cone}, that		
		\begin{equation*}
			\co((id,\tau_{\infty}))\sim \co(\tau_{\infty})\quad\mbox{and}\quad \co(\pi_{\infty})\sim \co(\hat{\pi}_{\infty})
		\end{equation*}
		where $X\sim Y$ means that the complexes $X$ and $Y$ are quasi-isomorphic. \\
		Let $C(M)= C^{\bullet}_{cont}(G_{F,S},M).$ Using Proposition \ref{pro cub},  the  commutative cubic diagram
		\begin{equation*}
			\xymatrix@=1pc{  C(M)\ar[rrrr]^-{res_{S}}\ar[dr]^-{id}\ar[ddddd]^-{id}&  &  &  & C_{f}\oplus C_{\infty}\ar[dr]^-{(id,\tau_{\infty})}\ar[ddddd]^-{\pi_{\infty}}\\
				&  C(M) \ar[rrrr]^-{\widehat{res}_{S}}\ar[ddddd]^-{id}&  &  &  & C_{f}\oplus \hat{C}_{\infty}\ar[ddddd]^-{\hat{\pi}_{\infty}}\\
				&  &   &   &\\
				&  &  &   &  \\
				&  &  &  &   \\
				C(M)\ar[rrrr]^-{\pi_{\infty}\circ \, res_{S}}\ar[dr]^-{id}&  &  &  &  C_{\infty}\ar[dr]^-{\tau_{\infty}}\\
				&   C(M)\ar[rrrr]^-{\widehat{\pi}_{\infty}\circ\, \widehat{res}_{S}}&  &  &  &  \hat{C}_{\infty}}
		\end{equation*}	
		gives rise to the commutative diagram
		\begin{equation*}
			\xymatrix@=1.5pc{\co(res_{S})\ar[r]\ar[d]& \co(\widehat{res}_{S})\ar[r]\ar[d]& \co((id,\tau_{\infty})\ar[d]^-{qis\sim}\ar[r]& \co(res_{S})[1]\ar[d]\\
				\co(\pi_{\infty}\circ res_{S})\ar[r]\ar[d]  & \co(\widehat{\pi}_{\infty}\circ\widehat{res}_{S})\ar[r]\ar[d]    &  \co(\tau_{\infty})\ar[r]  & \co(\pi_{\infty}\circ res_{S})[1]\\
				\co(\pi_{\infty})\ar[r]^-{qis\sim}\ar[d]& \co(\hat{\pi}_{\infty})\ar[d]&   &  \\
				\co(res_{S})[1]\ar[r]& \co(\widehat{res}_{S})[1]&    &    }
		\end{equation*}	
		Taking the cohomology, we obtain for all $n\in \Z$ an exact commutative diagram    	
		\begin{equation*}
			\xymatrix@=2pc{   &  H^{n}(\mathrm{Cone}(id,\tau_{\infty}))\ar[d]\ar[r]^-{\sim}& H^{n}(\mathrm{Cone}(\tau_{\infty}))\ar[d]&   \\
				H^{n}(\mathrm{Cone}(\pi_{\infty}))\ar[r]\ar[d]^-{\wr}& H^{n+1}(\mathrm{Cone}(res_{S}))\ar[r]\ar[d]& H^{n+1}(\mathrm{Cone}(\pi_{\infty}\circ res_{S}))\ar[r]\ar[d]& 	H^{n+1}(\mathrm{Cone}(\pi_{\infty}))\ar[d]^-{\wr}\\
				H^{n}(\mathrm{Cone}(\hat{\pi}_{\infty}))\ar[r]& H^{n+1}(\mathrm{Cone}(\widehat{res}_{S}))\ar[d]\ar[r]& H^{n+1}(\mathrm{Cone}(\hat{\pi}_{\infty}\circ\widehat{res}_{S}))\ar[r]\ar[d]& 	H^{n+1}(\mathrm{Cone}(\hat{\pi}_{\infty}))\\
				& H^{n+1}(\mathrm{Cone}((id,\tau_{\infty})))\ar[r]^-{\sim}&
				H^{n+1}(\mathrm{Cone}(\tau_{\infty}))&  	
			}
		\end{equation*}		
		In particular, for $n=-1$ we have a commutative diagram
		\begin{equation}\label{default Poitou-Tate}
			\xymatrix@=1pc{    & H^{-1}(\mathrm{Cone}(\tau_{\infty}))\ar[d]&   \\
				& H^{0}(\mathrm{Cone}(\pi_{\infty}\circ res_{S}))\ar[r]\ar[d]& 	H^{0}(\mathrm{Cone}(\pi_{\infty}))\ar[d]^-{\wr}\\
				& H^{0}(\mathrm{Cone}(\hat{\pi}_{\infty}\circ\widehat{res}_{S}))\ar[r]\ar[d]& 	H^{0}(\mathrm{Cone}(\hat{\pi}_{\infty}))\\
				& 	H^{0}(\mathrm{Cone}(\tau_{\infty}))=0&  	
			}
		\end{equation}
		and an exact sequence
		\begin{equation}\label{modified Poitou-Tate duality}
			\xymatrix@=1pc{ 	H^{-1}(\co(\hat{\pi}_{\infty}))\ar[r]& H^{0}(\mathrm{Cone}(\widehat{res}_{S}))\ar[r]& \reallywidehat{\sh}_{S_{f}}^{1,+}(M)	\ar[r]& 0,
			}
		\end{equation}	
		where
		\begin{equation*}
			\reallywidehat{\sh}_{S_{f}}^{1,+}(M):=\ker(H^{0}(\mathrm{Cone}(\hat{\pi}_{\infty}\circ\widehat{res}_{S}))\longrightarrow H^{0}(\mathrm{Cone}(\hat{\pi}_{\infty}))).
		\end{equation*}
		From \eqref{description totally positive cone} and \eqref{cone pi} we deduce that:
		\begin{equation*}
			\mbox{$\sh_{S_{f}}^{1,+}(M)$}=\xymatrix@=1.5pc{\ker (H^{0}(\mathrm{Cone}(\pi_{\infty}\circ res_{S}))\ar[r]& 	H^{0}(\mathrm{Cone}(\pi_{\infty}))).}
		\end{equation*}
		So,	the diagram \eqref{default Poitou-Tate} leads to the following exact sequence
		\begin{equation}\label{exact Z(M)}
			\xymatrix@=2pc{0\ar[r]& Z(M)\ar[r]& \sh_{S_{f}}^{1,+}(M)\ar[r]& \reallywidehat{\sh}_{S_{f}}^{1,+}(M)\ar[r]&0,}
		\end{equation}
		where  $Z(M):=\ker(H^{0}(\mathrm{Cone}(\pi_{\infty}\circ res_{S}))\longrightarrow  H^{0}(\mathrm{Cone}(\hat{\pi}_{\infty}\circ\widehat{res}_{S}))).$\\
		Recall that, for all $n\geq -1,$    $H^{n}(\mathrm{Cone}(\widehat{res}_{S})):=\hat{H}^{n+1}_{c,cont}(G_{F,S},M).$ Then, by \cite[Proposition 5.7.4.]{Nekovar06} we have
		\begin{equation*}
			H^{n}(\co(\widehat{res}_{S})))\cong H^{2-n}(G_{F,S},M^{\ast})^{\vee}.
		\end{equation*}
		Dualizing the exact sequence \eqref{modified Poitou-Tate duality}, the local duality  \eqref{local duality} and the isomorphism (\ref{cone pi}) give the following isomorphism
		\begin{equation}\label{global dual}
			\reallywidehat{\sh}_{S_{f}}^{1,+}(M)\cong \mbox{$\sh_{S_{f}}^{2}(M^{\ast})^{\vee}.$}
		\end{equation}
		Finally, according to the exact sequence \eqref{exact Z(M)} and the isomorphism \eqref{global dual} we deduce the desired result.
	\end{proof}
	
	In the next lemma, we give an explicit description of the module
	\begin{equation*}
		Z(M):=\ker(H^{0}(\mathrm{Cone}(\pi_{\infty}\circ res_{S}))\longrightarrow H^{0}(\mathrm{Cone}(\hat{\pi}_{\infty}\circ\widehat{res}_{S}))).
	\end{equation*}
	Firstly, recall that for any infinite place $v$ of $F$ a fixed extension $F^{sep}\hookrightarrow F_{v}^{sep}$ of the embedding  $F\hookrightarrow F_{v}$  defines a continuous homomorphism
	$$\xymatrix@=1pc{G_{F_{v}}\ar@{^{(}->}[r]& G_{F,S},}$$
	hence a restriction map  $res_{v}=(res_{v}^{i})$
	\begin{equation*}
		\xymatrix@=1.5pc{    C^{\bullet}_{cont}(G_{F,S},M)^{0}=M\ar[r]^-{d^{0}_{C(M)}}\ar[d]^-{res^{0}_{v}}&C^{\bullet}_{cont}(G_{F,S},M)^{1}\ar[d]^-{res^{1}_{v}}\ar[r]&\cdots\\ C^{\bullet}_{cont}(G_{F_{v}},M)^{0}=M\ar[r]^-{d^{0}_{v}}&C^{\bullet}_{cont}(G_{F_{v}},M)^{1}\ar[r]&\cdots}
	\end{equation*}
	Secondly, observe that
	\begin{equation}\label{res0}
		res^{0}_{v}(x)=x\quad\mbox{for all $x\in\ker d^{0}_{C(M)}=M^{G_{F,S}},$}
	\end{equation}
	since the restriction map in degree $0$ is the inclusion map
	\begin{equation*}
		\xymatrix@=2pc{ M^{G_{F,S}}\ar@{^{(}->}[r] & M^{G_{F_{v}}}.}
	\end{equation*}
	Moreover, if $f: C^{\bullet}_{cont}(G_{F,S},M) \longrightarrow C^{\bullet}_{cont}(G_{F_{v}},M)$ and $res_{v}$ are homotopic, then
	\begin{equation*}
		f^{0}(x)=res^{0}_{v}(x)=x \quad\mbox{for all } x\in\ker d^{0}_{C(M)}.
	\end{equation*}	
	Indeed, since  $f$ and $res_{v}$ are homotopic,  there exists a morphism
	\begin{equation*}
		h^{1}:C^{\bullet}_{cont}(G_{F,S},M)^{1}\longrightarrow C^{\bullet}_{cont}(G_{F_{v}},M)^{0}\;\mbox{such that}\;
		f^{0}-res^{0}_{v}=h^{1}d^{0}_{C(M)}.
	\end{equation*}
	Thus, $f^{0}(x)=res^{0}_{v}(x)$ for all $x\in\ker d^{0}_{C(M)}.$\vskip 6pt
	Now, since $\hat{\pi}_{\infty}\circ\widehat{res}_{S}=\tau_{\infty}\circ \pi_{\infty}\circ res_{S}$ and  $(\tau_{\infty})^{i}=id$ for all integer $i\geq 0,$ we get
	\begin{equation*}
		d^{0}_{\co(\hat{\pi}_{\infty}\circ\, \widehat{res}_{S})}=d^{0}_{\co(\pi_{\infty}\circ\,  res_{S})}=\begin{pmatrix}
			d^{0}_{C_{\infty}}& (\pi_{\infty}\circ res_{S})^{1}\\
			&  \\
			0& -d^{1}_{C(M)}
		\end{pmatrix},
	\end{equation*}
	\begin{equation*}
		d^{-1}_{\co(\hat{\pi}_{\infty}\circ\, \widehat{res}_{S})}=\begin{pmatrix}
			\bigoplus_{v\mid\infty}N_{v}&(\pi_{\infty}\circ res_{S})^{0} \\
			&   \\
			0& -d^{0}_{C(M)}
		\end{pmatrix}\quad\mbox{and}\quad
		d^{-1}_{\co(\pi_{\infty}\circ\, res_{S})}=\begin{pmatrix}
			0& (\pi_{\infty}\circ res_{S})^{0}\\
			&  \\
			0& -d^{0}_{C(M)}
		\end{pmatrix}
	\end{equation*}
	where $N_{v}=\sum_{\sigma\in G_{F_{v}}}\sigma.$ Further, remark that $(\pi_{\infty}\circ res_{S})^{0}=\bigoplus_{v\mid\infty}res^{0}_{v}.$\\
	Therefore, we have an exact commutative diagram
	\begin{equation}\label{description of Z(M)}
		\xymatrix@=2pc{ 0\ar[r]& \mathrm{Im}\,d^{-1}_{\co(\pi_{\infty}\circ\, res_{S})}\ar[d]\ar[r]& \ker d^{0}_{\co(\pi_{\infty}\circ\,  res_{S})}\ar@{=}[d]\ar[r]&H^{0}(\mathrm{Cone}(\pi_{\infty}\circ res_{S}))\ar[d]\ar[r]& 0\\
			0\ar[r]& \mathrm{Im}\,d^{-1}_{\co(\hat{\pi}_{\infty}\circ\, \widehat{res}_{S})}\ar[r]& \ker d^{0}_{\co(\hat{\pi}_{\infty}\circ\, \widehat{res}_{S})} \ar[r]&H^{0}(\mathrm{Cone}(\hat{\pi}_{\infty}\circ\widehat{res}_{S}))\ar[r]&0}
	\end{equation}
	Recall that $Z(M):=\ker(H^{0}(\mathrm{Cone}(\pi_{\infty}\circ res_{S}))\longrightarrow H^{0}(\mathrm{Cone}(\hat{\pi}_{\infty}\circ\widehat{res}_{S}))).$
	\begin{lem}\label{Lemma Z(M)}
		We have the following isomorphism of $\Z_{2}[[G_{F, S}]]$-modules
		\begin{equation*}
			Z(M)\cong (\oplus_{v\mid\infty}N_{v})(\hat{C}^{-1}_{\infty})/I_{\infty},
		\end{equation*}
		where $\displaystyle{I_{\infty}:=(\oplus_{v|\infty} res_{v}^{0})(\ker(d^{0}_{C(M)}))\bigcap \, (\oplus_{v\mid\infty}N_{v})(\hat{C}^{-1}_{\infty})}.$
	\end{lem}
	\begin{proof}
		The commutative diagram \eqref{description of Z(M)} shows that
		\begin{equation*}
			Z(M) \cong \mathrm{Im}\,d^{-1}_{\co(\hat{\pi}_{\infty}\circ\, \widehat{res}_{S})}/\mathrm{Im}\,d^{-1}_{\co(\pi_{\infty}\circ\, res_{S})}.
		\end{equation*}
		Since $C^{-1}_{\infty}=0,$ we can identify $(\co(\pi_{\infty} \circ res_{S}))^{0}$  with $C(M)^{0}.$ Consider the following commutative diagram
		\begin{equation}\label{diagram calcule}
			\xymatrix@=2pc{ 0\ar[r]& \ker d^{-1}_{\co(\pi_{\infty}\circ\, res_{S})}\ar[d]\ar[r]&  C(M)^{0} \ar@{^{(}->}[d]\ar[r]&\mathrm{Im}\,d^{-1}_{\co(\pi_{\infty}\circ\, res_{S})}\ar@{^{(}->}[d]\ar[r]& 0\\
				0\ar[r]& \ker d^{-1}_{\co(\hat{\pi}_{\infty}\circ\, \widehat{res}_{S})}\ar[r]& \hat{C}^{-1}_{\infty} \bigoplus C(M)^{0} \ar[r]& \mathrm{Im}\,d^{-1}_{\co(\hat{\pi}_{\infty}\circ\, \widehat{res}_{S})}\ar[r]&0}
		\end{equation}
		The kernel $\ker d^{-1}_{\co(\pi_{\infty}\circ\, res_{S})}$ is trivial. Indeed, let $y \in C(M)^{0}$ such that $d^{-1}_{\co(\pi_{\infty}\circ\, res_{S})}(0,y)=0.$ This means that
		\begin{equation*}
			\begin{cases}
				res^{0}_{v}(y)=0,\quad\mbox{for all $v\mid \infty$}\, ;\\
				
				y\in \ker d^{0}_{C(M)}.
			\end{cases}
		\end{equation*}
		Hence, by \eqref{res0}, we get $res^{0}_{v}(y)=y=0$ since $y\in \ker d^{0}_{C(M)}.$\\
		Therefore, the commutative diagram \eqref{diagram calcule} induces the following exact sequence
		\begin{equation*}
			\xymatrix@=2pc{ 0\ar[r]& \ker d^{-1}_{\co(\hat{\pi}_{\infty}\circ\, \widehat{res}_{S})} \ar[r]^{\qquad \pi}& \hat{C}^{-1}_{\infty} \ar[r]& \mathrm{Im}\,d^{-1}_{\co(\hat{\pi}_{\infty}\circ\, \widehat{res}_{S})}/\mathrm{Im}\,d^{-1}_{\co(\pi_{\infty}\circ\, res_{S})}\ar[r]& 0,}
		\end{equation*}
		where $\pi$ is the canonical projection.\\
		Let $\varphi$ be the restriction of the map
		\begin{equation*}
			\xymatrix@=2pc{ \oplus_{v|\infty}N_{v}: \, \hat{C}^{-1}_{\infty} \ar[r]& \hat{C}^{0}_{\infty}}
		\end{equation*}
		to $\pi(\ker d^{-1}_{\co(\hat{\pi}_{\infty}\circ\,\widehat{res}_{S})}).$ We set
		\begin{equation*}
			I_{\infty}:= (\oplus_{v|\infty} res_{v}^{0})(\ker d^{0}_{C(M)})\bigcap\, (\oplus_{v|\infty}N_{v})(\hat{C}^{-1}_{\infty}).
		\end{equation*}
		We claim that the image of $\varphi$ is $I_{\infty}.$		
		Indeed, let $(x_{v})_{v|\infty}$ be an element of $\mathrm{im} \varphi.$ Since $\varphi$ is the restriction of $\oplus_{v|\infty} N_{v},$ it is clear that $(x_{v})_{v|\infty}$ belongs to $(\oplus_{v|\infty}N_{v})(\hat{C}^{-1}_{\infty}).$ Further, there exists $((a_{v})_{v|\infty}, y)$ in $\ker d^{-1}_{\co(\hat{\pi}_{\infty}\circ\,\widehat{res}_{S})}$ such that $\varphi((a_v)_{v|\infty})= (x_{v})_{v|\infty}.$ In particular, we get that $ (x_{v})_{v|\infty}$ is an element of $(\oplus_{v|\infty} res_{v}^{0})(\ker d^{0}_{C(M)}).$\\		
		Let's prove the other inclusion. For all $(x_{v})_{v|\infty}$ in $I_{\infty},$ there exist $y$ in $\ker d^{0}_{C(M)}$ and $(a_{v})_{v|\infty}$ in $\hat{C}^{-1}_{\infty}$ such that for each infinite place $v$
		\begin{equation*}
			x_{v}=res_{v}^{0}(y)=N_{v}(a_{v}).
		\end{equation*}
		Then $((a_{v})_{v|\infty},-y)$ belongs to $\ker d^{-1}_{\co(\hat{\pi}_{\infty}\circ\, \widehat{res}_{S})}$ and $\varphi(\pi(((a_{v})_{v|\infty},-y)))= (x_{v})_{v|\infty}.$\\
		These lead to the following commutative diagram
		\begin{equation*}
			\xymatrix@=1.5pc{
				0\ar[r]& \ker d^{-1}_{\co(\hat{\pi}_{\infty}\circ\, \widehat{res}_{S})} \ar@{>>}[d]^{\bar{\varphi}:=\varphi\circ \pi}\ar[r]^{\qquad \pi}& \hat{C}^{-1}_{\infty}  \ar[d]^{\oplus_{v|\infty}N_{v}}\ar[r]& \mathrm{Im}\,d^{-1}_{\co(\hat{\pi}_{\infty}\circ\, \widehat{res}_{S})}/\mathrm{Im}\,d^{-1}_{\co(\pi_{\infty}\circ\, res_{S})}\ar[d]\ar[r]&0\\
				0\ar[r]& I_{\infty} \ar[r]&  \hat{C}_{\infty}^{0} \ar[r]& \hat{C}^{0}/I_{\infty} \ar[r]& 0}
		\end{equation*}
		Observe that a pair $((a_{v})_{v|\infty}, y)$ in $\ker d^{-1}_{\co(\hat{\pi}_{\infty}\circ\, \widehat{res}_{S})}$ belongs to $\ker \bar{\varphi}$ exactely when $ (a_{v})_{v|\infty}$ is an element of $\ker( \oplus_{v|\infty}N_{v})$ and $y=0.$ In particular, $\pi$ induces an isomorphism between $\ker \bar{\varphi}$ and $\ker (\oplus_{v|\infty}N_{v}).$ Hence we obtain the following exact sequence
		\begin{equation*}
			\xymatrix@=1.5pc{ 0 \ar[r] &  \mathrm{Im}\,d^{-1}_{\co(\hat{\pi}_{\infty}\circ\, \widehat{res}_{S})}/\mathrm{Im}\,d^{-1}_{\co(\pi_{\infty}\circ\, res_{S})} \ar[r] & \hat{C}^{0}/I_{\infty} \ar[r]&  \hat{C}^{0}/(\oplus_{v|\infty}N_{v})(\hat{C}^{-1}_{\infty}) \ar[r]&  0.}
		\end{equation*}
		Therefore, $\mathrm{Im}\,d^{-1}_{\co(\hat{\pi}_{\infty}\circ\, \widehat{res}_{S})}/\mathrm{Im}\,d^{-1}_{\co(\pi_{\infty}\circ\, res_{S})}$ is isomorphic to
		\begin{equation*}
			\ker( \hat{C}^{0}/I_{\infty} \longrightarrow \hat{C}^{0}/(\oplus_{v|\infty}N_{v})(\hat{C}^{-1}_{\infty})),
		\end{equation*}
		which is none other than $(\oplus_{v|\infty}N_{v})(\hat{C}^{-1}_{\infty})/I_{\infty}.$ Then
		\begin{equation*}
			Z(M)\cong (\oplus_{v|\infty}N_{v})(\hat{C}^{-1}_{\infty})/I_{\infty}.
		\end{equation*}
	\end{proof}
	Summarizing, we get the main result of this paper namely, the global Poitou-Tate duality for totally positive Galois cohomology.
	\begin{theo}\label{theorem1}
		Let $M$ be a discrete or compact $\Z_{2}[[G_{F,S}]]$-module.
		\begin{enumerate}[label=(\roman*)]
			\item 	There is a perfect pairing
			\begin{equation*}
				\begin{array}{ccccc}
					\sh_{S_{f}}^{2,+}(M) & \times & \sh_{S_{f}}^{1}(M^{\ast}) & \xymatrix@=1.5pc{ \ar[r]&} &
					\Q_{2}/\Z_{2}.
				\end{array}
			\end{equation*}
			\item We have an exact sequence	
			\begin{equation*}
				\xymatrix@=2pc{0\ar[r]& (\oplus_{v|\infty}N_{v})(\hat{C}^{-1}_{\infty})/I_{\infty} \ar[r]&
					\sh_{S_{f}}^{1,+}(M)\ar[r]& \sh_{S_{f}}^{2}(M^{\ast})^{\vee}\ar[r]&0,}
			\end{equation*}
			where $\displaystyle{I_{\infty}=(\oplus_{v|\infty} res_{v}^{0})(\ker d^{0}_{C(M)})\bigcap\, (\oplus_{v|\infty}N_{v})(\hat{C}^{-1}_{\infty})}.$
		\end{enumerate} \hfill $\square$
	\end{theo}
	
	Let us compute the module $Z(M)$ for the compact $\Z_{2}[[G_{F, S}]]$-module $M=\Z_{2}(i),$ the twist "à la Tate" of the ring of dyadic integers. Recall that
	\begin{equation*}
		Z(M) \cong (\oplus_{v|\infty}N_{v})(\hat{C}^{-1}_{\infty})/I_{\infty}
	\end{equation*}
	and notice that the kernel $\ker (d^{0}_{C(M)})=M^{G_{F,S}}.$ Since $\Z_{2}^{G_{F,S}}=\Z_{2}$ and $\Z_{2}(i)^{G_{F,S}}=0$ if $i\neq 0,$ we get that
	\begin{equation*}
		I_{\infty} = \begin{cases}
			(\oplus_{v|\infty}N_{v})(\hat{C}^{-1}_{\infty}), \, \mbox{ if } i=0;\\
			\\
			0,\quad\qquad\qquad \quad \mbox{ if } i\neq 0.
		\end{cases}	
	\end{equation*}
	Further, remark that for any infinite place $v,$ the group $G_{F_{v}}$ is either trivial or equals to $\{1,\sigma\}$ according to $v$ is complex or real respectively,  where $\sigma$ is the complex conjugation. Hence if $v$ is complex, $N_{v}$ coincides with the identity of $\Z_{2}(i).$ If $v$ is real, $N_{v}$ is the multiplication by $1+\sigma.$ Therefore,
	\begin{equation*}
		Z(\Z_{2}(i))\cong \begin{cases}
			\qquad 0, \qquad \qquad\qquad \quad \qquad\quad\, \mbox{ if } i=0;\\\
			\underset{v \: is \: complex}{\bigoplus} \Z_{2}(i) \underset{v \: is \: real}{\bigoplus} 2\Z_{2}(i) ,\quad \mbox{ if $i\neq 0$ is even;} \\\
			\underset{v \: is \: complex}{\bigoplus} \Z_{2}(i), \quad\qquad \,\quad\qquad \mbox{ if $i$ is odd.}
		\end{cases}	
	\end{equation*}
	
	As a consequence of Theorem \ref{theorem1}  we then obtain the following:
	\begin{coro}\label{M=Z(i)}
		\begin{itemize}
			\item [(i)] We have an isomorphism
			\begin{center}$\sh_{S_{f}}^{2,+}(\mathbb{Z}_{2}(i))\cong \sh_{S_{f}}^{1}(\mathbb{Q}_{2}/\Z_{2}(1-i))^{\vee}.$\end{center}
			\item [(ii)] We have an exact sequence
			\begin{equation*}
				\xymatrix@=2pc{0\ar[r]& Z(\Z_{2}(i))\ar[r]& \sh_{S_{f}}^{1,+}(\mathbb{Z}_{2}(i))\ar[r]& \sh_{S_{f}}^{2}(\mathbb{Q}_{2}/\mathbb{Z}_{2}(1-i))^{\vee}\ar[r]&0.}
			\end{equation*}
			In particular, if $i=0$ or $i$ is odd and $F$ is a totally real number field, we have a perfect pairing
			\begin{equation*}
				\begin{array}{ccccc}
					\sh_{S_{f}}^{1,+}(\Z_{2}(i)) & \times & \sh_{S_{f}}^{2}(\Q_{2}/\Z_{2}(1-i)) & \xymatrix@=1.5pc{ \ar[r]&} &
					\Q_{2}/\Z_{2}.
				\end{array}
			\end{equation*}
		\end{itemize}
	\end{coro}
	\hfill $\square$\vskip 6pt
	As application, we give a description of $\sh_{S_{f}}^{2,+}(\mathbb{Z}_{2}(i))$ in terms of an Iwasawa module.\\
	Let $\mathcal{O}_{F,S}$ be the ring of $S$-integers of a number field $F.$ For $i\geq 1$ an integer, the positive \'{e}tale wild kernel (\cite[Definition 2.2]{AAM21}) is the group
	\begin{equation*}
		WK_{2i-2}^{\mbox{\'{e}t},+}\mathcal{O}_{F,S}:=\mbox{$\sh_{S_{f}}^{2,+}(\mathbb{Z}_{2}(i))$}=\xymatrix@=2pc{\ker(H^{2}_{+}(G_{F,S},\Z_{2}(i))\ar[r]& \bigoplus_{v\in S_{f}}H^{2}(F_{v},\Z_{2}(i))).}
	\end{equation*}
	For i =1, the group $WK_{2i-2}^{\mbox{\'{e}t},+}\mathcal{O}_{F,S}$ is isomorphic to the $2$-part of the narrow
	$S$-class group $A^{+}_{F,S}$ (see also Remark $2.3$ in [2]). In particular it depends on the set $S.$\\
	
	Let $F_{\infty}:=\underset{n\geq 0}{\bigcup} F_{n}$ be  the cyclotomic	$\Z_{2}$-extension of $F$ with Galois group  $\Gamma=\mathrm{Gal}(F_{\infty}/F),$ and let
	$X^{\prime, +}_{\infty}$ be the Galois group of the maximal $2$-extension of $F_{\infty},$ which is unramified at finite places and completely decomposed at all primes above $2.$	In particular,  we have $X^{\prime, +}_{\infty}=\underset{\longleftarrow}{\lim}A^{+}_{F_{n},S}.$ The following proposition is an analogue of Schneider's description of the classical  \'{e}tale wild kernel (\cite{Sc79}). 
	\begin{pro}\label{Iwasawa description}
		Let $i\geq 2$ be an integer. If either $i$ is odd, or $i$ is even and $\sqrt{-1}\in F,$ then
		\begin{equation*}
			WK_{2i-2}^{\mbox{\'{e}t},+}\mathcal{O}_{F,S}\cong X^{\prime, +}_{\infty}(i-1)_{\Gamma}.
		\end{equation*}
		In particular, in both cases we recover that the group $WK_{2i-2}^{\mbox{\'{e}t},+}\mathcal{O}_{F,S}$ is
		independent of the set $S$ containing the infinite and dyadic places of $F.$
	\end{pro}	
	\begin{proof} Let $j=1-i.$ 	Consider  the following  exact commutative diagram
		\begin{equation*}
			\xymatrix@=1pc{
				H^{1}(\Gamma,\Q_{2}/\Z_{2}(j))\ar@{^{(}->}[r]\ar[d]&
				H^{1}(G_{F,S},\Q_{2}/\Z_{2}(j))\ar[r]\ar[d]&
				H^{1}(G_{F_{\infty},S},\Q_{2}/\Z_{2}(j))^{\Gamma}\ar[d]\ar[r]&0\\
				\bigoplus_{v\in
					S_{f}}H^{1}(\Gamma_{v},\Q_{2}/\Z_{2}(j))\ar@{^{(}->}[r]&
				\bigoplus_{v\in
					S_{f}}H^{1}(F_{v},\Q_{2}/\Z_{2}(j))\ar[r]&\bigoplus_{v\in
					S_{f}}H^{1}(F_{v,\infty},\Q_{2}/\Z_{2}(j))^{\Gamma_{v}}\ar[r]&0}
		\end{equation*}
		where	$F_{v,\infty}$ is the cyclotomic $\Z_{2}$-extension of $F_{v}$ and
		$\Gamma_{v}=\mathrm{Gal}(F_{v,\infty}/F_{v})$ is the decomposition group of $v$ in 	$F_{\infty}/F.$
		By \cite[Lemma 3.1]{AAM21}, we have
		\begin{equation*}
			H^{1}(\Gamma,\Q_{2}/\Z_{2}(j))=0\,\,\mbox{and}\,\,H^{1}(\Gamma_{v},\Q_{2}/\Z_{2}(j))=0\;\;\mbox{for all $v\in S_{f}$},
		\end{equation*}
		and then
		\begin{eqnarray*}
			\mbox{$\sh_{S_{f}}^{1}(\Q_{2}/\Z_{2}(j))$} &=&\xymatrix@=1.5pc{\ker( H^{1}(G_{F_{\infty},S},\Q_{2}/\Z_{2}(j))^{\Gamma}
				\ar[r]& \bigoplus_{v\in
					S_{f}}H^{1}(F_{v,\infty},\Q_{2}/\Z_{2}(j))^{\Gamma})}  \\
			&=&\mathrm{Hom}(X^{\prime, +}_{\infty},
			\Q_{2}/\Z_{2})(-j)^{\Gamma}.
		\end{eqnarray*}
		Hence, using Theorem \ref{theorem1}, we obtain the isomorphism
		\begin{equation*}
			WK_{2i-2}^{\mbox{\'{e}t},+}\mathcal{O}_{F,S}\cong X^{\prime, +}_{\infty}(i-1)_{\Gamma}.
		\end{equation*}
		In particular, the group $WK_{2i-2}^{\mbox{\'{e}t},+}\mathcal{O}_{F,S}$ is
		independent of the set $S$ containing the infinite and dyadic places of $F.$ From now on, we adopt the   following notation $$WK_{2i-2}^{\mbox{\'{e}t},+}\mathcal{O}_{F,S}:=  WK_{2i-2}^{\mbox{\'{e}t},+}F.$$	
	\end{proof}	
	Let  $F_{\infty}=\cup_{n}F_{n}$ be the cyclotomic $\Z_{2}$-extension of $F$ and for
	$n\geq 0,$  $G_{n}=\mathrm{Gal}(F_{n}/F).$ Since for all $n\geq 0$
	\begin{equation*}
		(X^{\prime, +}_{\infty}(i-1)_{\Gamma_{n}})_{G_{n}} \cong X^{\prime, +}_{\infty}(i-1)_{\Gamma},
	\end{equation*}
	the above description of the positive \'{e}tale wild kernel, shows immediately that the positive \'{e}tale wild kernel satisfies Galois co-descent in the cyclotomic tower (see as well \cite[Corollary $3.3$]{AAM21}).
	\begin{coro}\label{corrollary codescent}
		If either $i$ is odd, or $i$ is even and $\sqrt{-1}\in F,$ then the positive \'{e}tale wild kernel satisfies Galois co-descent in the cyclotomic $\Z_{2}$-extension:
		\begin{equation*}
			(WK_{2i-2}^{\mbox{\'{e}t},+}F_{n})_{G_{n}}\cong 	WK_{2i-2}^{\mbox{\'{e}t},+}F.
		\end{equation*}
	\end{coro}
	\begin{rem}
		For a number field $F,$ we now that $WK_{2}^{\mbox{\'{e}t},+}F\cong WK_{2}^{\mbox{\'{e}t}}F$ (cf. \cite[Proposition 2.4]{AAM21}). Furthermore, if $\sqrt{-1}\in F$ the above result has been proved in \cite[Theorem 2.18]{KoMo}.
	\end{rem}
	\bibliographystyle{plain}	
	\bibliography{PTDPC}{}

\end{document}